\documentclass{article}
\usepackage[utf8]{inputenc}
\usepackage{amsmath}
\usepackage{amssymb}
\usepackage{textcomp}
\usepackage{hyperref}
\usepackage[english]{babel}
\usepackage{amsthm}
\theoremstyle{definition}
\newtheorem{definition}{Definition}[section]
\newtheorem*{remark}{Remark}
\newtheorem*{notation}{Notation}
\newtheorem{theorem}{Theorem}[section]

\newtheorem{corollary}{Corollary}[theorem]
\newtheorem{claim}[theorem]{Claim}
\newtheorem{fact}[theorem]{Fact}
\newtheorem*{acknowledgements}{Acknowledgements}
\title{A note on strong Erd\H{o}s-Hajnal for graphs with bounded VC-minimal complexity}
\author{Yayi Fu }
\date{}
\begin{document}
\maketitle
\begin{abstract}
 Inspired by Adler's idea on VC minimal theories \cite{adler2008theories}, we introduce VC-minimal complexity. 
 We show that for any $N\in\mathbb{N}^{>0}$,
 there is $k_N>0$ such that for any finite bipartite graph $(X,Y;E)$
with VC-minimal complexity $< N$, there exist $X'\subseteq X$, $Y'\subseteq Y$ with $|X'|\geq k_N |X|$, $|Y'|\geq k_N |Y|$ such that $X'\times Y' \subseteq E$ or $X'\times Y'\cap E=\emptyset$. 
\end{abstract} \hspace{10pt}
\section{Introduction}
\indent 

\emph{Erd\H{o}s-Hajnal conjecture} \cite{erdos1989ramsey} says for any graph $H$ there is $\epsilon>0$ such that if a graph $G$ 
does not contain any induced subgraph isomorphic to $H$ then $G$ has a clique or an anti-clique of size $\geq |G|^\epsilon$. 
More generally, we say a family of finite graphs has the \emph{Erd\H{o}s-Hajnal property} 
if there is $\epsilon>0$ such that for any graph $G$ in the family, $G$ has a clique or an anti-clique of size $\geq|G|^\epsilon$.
A family of finite graphs has the \emph{strong Erd\H{o}s-Hajnal property} if there is 
$\epsilon>0$ such that for any graph $G=(V,E)$ in the family,
there exist $X,Y\subseteq V$ such that $X\cap Y=\emptyset$, $|X|\geq \epsilon|V|$, 
$|Y|\geq \epsilon|V|$, and $X\times Y\subseteq E$ or $X\times Y\subseteq\neg E$. 
Strong Erd\H{o}s-Hajnal property implies Erd\H{o}s-Hajnal property. 
(See \cite[Theorem~1.2.]{alon2005crossing}.) 
Malliaris and Shelah proved in \cite{malliaris2014regularity} that the family 
of stable graphs has the Erd\H{o}s-Hajnal property. 
Chernikov and Starchenko gave another proof for stable graphs in \cite{chernikov2018note} and 
in \cite{chernikov2018regularity} they proved that the family of distal graphs has the strong Erd\H{o}s-Hajnal property. 
In general, we are interested in whether the family of finite VC dimension (i.e. NIP \cite{simon2015guide}) graphs,
which contains both stable graphs and distal graphs, has the Erd\H{o}s-Hajnal property.
Motivation for studying this problem was given in \cite{fox2019erdHos}, 
which also gave a lower bound $e^{(\log n)^{1-o(1)}}$ for largest clique or anti-clique in a graph with bounded VC dimension. 
In this paper, we consider graphs of bounded VC-minimal complexity, a special case of NIP graphs. 
Roughly speaking, we say a bipartite graph $(X,Y;E)$ has VC-minimal complexity $<N$ if for all $a\in X$, 
the set $\{y\in Y: (a,y)\in E\}$ is a finite union of Swiss Cheeses such that the sum of the number of holes and the number of Swiss Cheeses is $<N$.
We will show that strong Erd\H{o}s-Hajnal property holds for the family of finite bipartite graphs $(X,Y;E)$ of bounded VC-minimal complexity. 
One example is definable relations $E(x,y)$ with $|x|=1,|y|=1$ in $ACVF$ (algebraically closed valued field). 
Since $ACVF$ allows Swiss Cheese decomposition \cite{holly1995canonical}, 
given any $\mathcal{M}\models ACVF$ and any definable relation $E \subseteq M\times M$, 
the family $\{(X,Y;E_{\upharpoonright X\times Y}):X,Y$ finite subsets of $M\}$ has bounded VC-minimal complexity, 
and thus strong Erd\H{o}s-Hajnal property holds. 
This partially generalizes \cite[Example~4.11.(2)]{chernikov2018regularity}. \\
\indent 
We will prove the following: 
\begin{theorem}
     For $N>0$, let $k_N=\dfrac{1}{2^{N+4}}$. If a finite bipartite graph $(X,Y;E)$ has VC-minimal complexity $<N$ then there exist $X'\subseteq X$, $Y'\subseteq Y$ with $|X'|\geq k_N|X|$, $|Y'|\geq k_N|Y|$ such that $X'\times Y'\subseteq E$ or $X'\times Y'\cap E=\emptyset$.
\end{theorem}
\noindent
\begin{acknowledgements}
The author is grateful to her advisor Sergei Starchenko for helpful suggestions.
\end{acknowledgements}
\section{Preliminaries}
\indent

 The following definitions \ref{directedfam}, \ref{swisscheese}, \ref{vcmincom} are based on notions in \cite{adler2008theories}. 
\begin{definition}\label{directedfam}
Given a set $U$, a family of subsets $\Psi=\{B_i:i\in I\}\subseteq\mathcal{P}(U)$, where $I$ is some index set, is called a \emph{directed family} if for any $B_i,B_j\in\Psi$, $B_i\subseteq B_j$ or $B_j \subseteq B_i$ or $B_i\cap B_j=\emptyset$.
\end{definition}
\begin{definition}\label{swisscheese}
 Given a directed family $\Psi$ of subsets of $U$, a set $B\in \Psi$ is a called a \emph{$\Psi$-ball}. A set $S \subseteq U$ is a \emph{$\Psi$-Swiss cheese} if $S = B \setminus (B_0 \cup ... \cup B_n)$, where each of $B,B_0,...,B_n$ is a $\Psi$-ball. We will call $B$ an \emph{outer ball} of $S$, and each $B_i$ is called a \emph{hole} of $S$.
\end{definition}
\begin{definition}
    A \emph{graph} $G$ is a pair $(V,E)$ where $V$ is a finite set of vertices and $E\subseteq V\times V$ is a binary symmetric anti-reflexive relation. 
    \end{definition}
\begin{definition}
A \emph{bipartite graph} is a triple $(X,Y;E)$ where $X$, $Y$ are finite sets, $X\cap Y=\emptyset$ and $E\subseteq X\times Y$ a symmetric relation.
\end{definition}
\begin{notation}
 Given a bipartite graph $(X,Y;E)$, $a\in X$, $S\subseteq Y$, we define $E(a,S)$ as the set $\{b\in S: (a,b)\in E\}$.   
\end{notation}
\begin{definition}\label{vcmincom}
    Given a finite bipartite graph $(X,Y;E)$, we say it has \emph{VC-minimal complexity $<N$} if there is a directed family $\Psi$ of subsets of $Y$ such that for each $a\in X$, $E(a,Y)$ is a finite disjoint union of $\Psi$-Swiss cheeses and the number of outer balls $+$ the number of holes $<N$. i.e. if $E(a,Y)=(B_{11}\setminus (B_{12}\cup...\cup B_{1d(1)}))$ $\dot\cup$ $...$ $\dot\cup$ $(B_{s1}\setminus (B_{s2}\cup...\cup B_{sd(s)}))$ then $d(1)+...+d(s)<N$.
    \end{definition}
\section{Proof}
\begin{theorem}
\label{mainthm}
 For $N>0$, let $k_N=\dfrac{1}{2^{N+4}}$. If a finite bipartite graph $(X,Y;E)$ has VC-minimal complexity $<N$ then there exist $X'\subseteq X$, $Y'\subseteq Y$ with $|X'|\geq k_N|X|$, $|Y'|\geq k_N|Y|$ such that $X'\times Y'\subseteq E$ or $X'\times Y' \cap E=\emptyset$.
\end{theorem}
\begin{proof}Fix a directed family $\Psi$ for $(X,Y;E)$.\\
\indent
     We prove by induction on $N$. If $N=1$ then for all $a\in X$, $E(a,Y)=\emptyset$. So $X\times Y\subseteq\neg E$.\\
\indent
Suppose true for $N$ and we show for $N+1$.
\\
\indent
Let $(X,Y;E)$ be a finite biipartiite graph with VC-minimal complexity $<N+1$.
Then there is a directed family $\Psi$ such that for each $a\in X$,  
\begin{equation*}
    E(a,Y)=(B^a_{11}\setminus (B^a_{12}\cup...\cup B^a_{1d(1)}))\dot\cup... \dot\cup(B^a_{s_a1}\setminus (B^a_{s_a2}\cup...\cup B^a_{s_ad(s_a)}))
\end{equation*}
where the $B^a_{kl}$'s are $\Psi$-balls and $d(1)+...+d(s_a)<N+1$.
Consider the finite family \begin{equation*}
    \mathcal{F}:=\{B^a_{kl}:a\in X, k,l\in\mathbb{N}\}\cup\{Y\}.
    \end{equation*}
    Since $\mathcal{F}$ is finite and $|Y|\geq\frac{1}{8}|Y|$, 
    there is a minimal $Z\in\mathcal{F}$ such that $|Z|\geq\frac{1}{8}|Y|$ 
    (minimal with respect to the partial order $\subseteq$). 
    Let 
    \begin{equation*}
        \mathcal{F}':=\{B^a_{kl}:a\in X, k,l\in\mathbb{N}, B^a_{kl}\subsetneq Z\}. 
        \end{equation*}
        Let $C_1,...,C_m$ be maximal elements in $\mathcal{F}'$. Then $\forall a\in X$, $\forall k,l \in\mathbb{N}$, $\forall t\in\{1,...,m\}$, if $B^a_{kl}\subsetneq Z$ then $B^a_{kl}\cap C_t=\emptyset$ or $B^a_{kl}\subseteq C_{t}$. Let $R=Z\setminus (C_1\cup...\cup C_{m})$.
\begin{claim}   \label{claim}
 $\forall a\in X$, $E(a,R)=R$ or $E(a,R)=\emptyset$.
 \end{claim}
\begin{proof}
     $E(a,Y)=(B^a_{11}\setminus (B^a_{12}\cup...\cup B^a_{1d(1)}))$ $\dot\cup$ $...$ $\dot\cup$ $(B^a_{s_a1}\setminus (B^a_{s_a2}\cup...\cup B^a_{s_ad(s_a)}))$.
     Suppose $E(a,R)\neq\emptyset$. 
     Then for some $k\in\{1,...,s_a\}$, 
\begin{equation*}
    (B^a_{k1}\setminus (B^a_{k2}\cup...\cup B^a_{kd(k)}))\cap R\neq\emptyset.
    \end{equation*}
     May assume $(B^a_{11}\setminus (B^a_{12}\cup...\cup B^a_{1d(1)}))\cap R\neq\emptyset$. 
     So $B^a_{11}\cap Z\neq\emptyset$. 
     Since $Z$ is $Y$ or a $\Psi$-ball, 
     $B^a_{11}\subsetneq Z$ or 
     $B^a_{11}\supseteq Z$. 
     If $B^a_{11}\subsetneq Z$ then $B^a_{11}\subseteq C_1\cup...\cup C_{m}$ and $B^a_{11}\cap R=\emptyset$,
     a contradiction. 
     Hence $B^a_{11}\supseteq Z$. 
     Similarly, for any hole $K\in\{B^a_{12},...,B^a_{1d(1)}\}$, 
     if $K\cap R\neq\emptyset$, 
     then $K\subsetneq Z$ or $K\supseteq Z$. 
     If $K\subsetneq Z$ then $K\subseteq C_1\cup...\cup C_{m}$ and $K\cap R=\emptyset$, a contradiction. 
     If $K\supseteq Z$,
     then 
     \begin{equation*}
         (B^a_{11}\setminus (B^a_{12}\cup...\cup B^a_{1d(1)}))\cap R\subseteq (B^a_{11}\setminus (B^a_{12}\cup...\cup B^a_{1d(1)}))\cap Z=\emptyset,
         \end{equation*}
         a contradiction. 
     Hence we must have $Z\subseteq B^a_{11}$ and $K\cap R=\emptyset$ for all $K\in\{B^a_{12},...,B^a_{1d(1)}\}$.
     So $R\subseteq B^a_{11}\setminus (B^a_{12}\cup...\cup B^a_{1d(1)})\subseteq E(a;Y)$.
     \end{proof}
\indent
Since $R\cup C_1\cup ...\cup C_m=Z$ and $|Z|\geq \frac{1}{8}|Y|$, by claim \ref{claim}, we may assume that $|C_1\cup...\cup C_{m}|\geq\frac{1}{16}|Y|$.\\
\indent
Let $t_0$ be smallest such that $|C_1\cup...\cup C_{t_0}|\geq\frac{1}{32}|Y|$. Because $|C_1\cup...\cup C_{t_0-1}|<\frac{1}{32}|Y|$ and $|C_{t_0}|<\frac{1}{8}|Y|$ (by minimality of $Z$), 
\begin{equation*}
    \frac{1}{32}|Y|\leq|C_1\cup...\cup C_{t_0}|\leq(\frac{1}{32}+\frac{1}{8})|Y|.
    \end{equation*}
    Let $C:=C_1\cup...\cup C_{t_0}$.\\
\indent
Consider 
\begin{equation*}
    A_1:=\{a\in X:\exists k,l\in\mathbb{N},
    B^a_{kl}\subseteq C\}, 
    \end{equation*}
    \begin{equation*}
        A_2:=\{a\in X:\forall k,l\in\mathbb{N},
        B^a_{kl}\nsubseteq C\}. 
        \end{equation*}
        \indent
Since $A_1\cup A_2=X$,
we have $|A_1|\geq\frac{1}{2}|X|$ or $|A_2|\geq\frac{1}{2}|X|$.\\
\indent
Suppose $|A_1|\geq\frac{1}{2}|X|$. For $a\in A_1$, 
\begin{equation*}
    E(a,Y\setminus C)=((B^a_{11}\setminus B^a_{12}\cup...\cup B^a_{1d(1)})\cap(Y\setminus C))\dot\cup...
    \end{equation*}
    \begin{equation*}
    \dot\cup((B^a_{s_a1}\setminus B^a_{s_a2}\cup...\cup B^a_{s_ad(s_a)})\cap (Y\setminus C))
    \end{equation*}
    \begin{equation*}
        =(((B^a_{11}\cap (Y\setminus C))\setminus (B^a_{12}\cap (Y\setminus C))\cup...\cup (B^a_{1d(1)}\cap(Y\setminus C)))\dot\cup...
        \end{equation*}
        \begin{equation*}
\dot\cup(B^a_{s_a1}\cap(Y\setminus C))\setminus (((B^a_{s_a2}\cap(Y\setminus C))\cup...\cup (B^a_{s_ad(s_a)}\cap(Y\setminus C)))).
\end{equation*}
Since $a\in A_1$, for some $B^a_{kl}$, $B^a_{kl}\subseteq C$. If $B^a_{kl}$ is an outer ball, say $B^a_{kl}=B^a_{11}$, then
\begin{equation*}
    E(a,Y\setminus C)=((B^a_{21}\cap (Y\setminus C))\setminus (B^a_{22}\cap (Y\setminus C))\cup...\cup (B^a_{2d(2)}\cap(Y\setminus C)))\dot\cup...
    \end{equation*}
    \begin{equation*}
        \dot\cup
    (B^a_{s_a1}\cap(Y\setminus C))\setminus (((B^a_{s_a2}\cap(Y\setminus C))\cup...\cup (B^a_{s_ad(s_a)}\cap(Y\setminus C)))).
    \end{equation*}
    If $B^a_{kl}$ is a hole, say $B^a_{kl}=B^a_{12}$, then
    \begin{equation*}
        E(a,Y\setminus C)=((B^a_{11}\cap (Y\setminus C))\setminus (B^a_{13}\cap (Y\setminus C))\cup...\cup (B^a_{1d(1)}\cap(Y\setminus C)))\dot\cup...
        \end{equation*}
        \begin{equation*}
\dot\cup(B^a_{s_a1}\cap(Y\setminus C))\setminus (((B^a_{s_a2}\cap(Y\setminus C))\cup...\cup (B^a_{s_ad(s_a)}\cap(Y\setminus C)))). 
    \end{equation*}
    Hence $(A_1, Y\setminus C; E)$ is a bipartite graph of VC-minimal complexity $<N$ such that for any $a\in A_1$, 
    $E(a,Y\setminus C)$ is a disjoint union of $\Psi'$-Swiss cheeses,
    where 
    \begin{equation*}
        \Psi':=\{D\cap (Y\setminus C): D\in\Psi\}.
        \end{equation*}
    By inductive hypothesis, there exist $F\subseteq A_1$, $G\subseteq Y\setminus C$ with
    \begin{equation*}
  |F|\geq k_N|A_1|\geq \frac{1}{2}k_N|X|,
  \end{equation*}
  \begin{equation*}
      |G|\geq k_N|Y\setminus C|\geq (1-\frac{1}{32}-\frac{1}{8})k_N|Y|
      \end{equation*}
      such that $F\times G\subseteq E$ or $F\times G\subseteq\neg E$. So the conclusion holds for $N+1$.\\
\indent
Suppose $|A_2|\geq\frac{1}{2}|X|$. 
Then $\forall a\in A_2$, $\forall k,l\in\mathbb{N}$,
$B^a_{kl}\nsubseteq C$.
\begin{claim}
    $\forall a\in A_2$, $E(a,C)=\emptyset$ or $E(a,C)=C$ by defintion of $A_2
$.
\end{claim}
\begin{proof}
We first show that for all $a\in A_2$,
for all $k,l\in\mathbb{N}$,
if $B^a_{kl}\cap C\neq \emptyset$
then $C\subseteq B^a_{kl}$.
\\
\indent
Fix $a\in A_2$ and $k,l\in\mathbb{N}$.    
If $B^a_{kl}\cap C\neq\emptyset$, 
then $B^a_{kl}\cap Z\neq\emptyset$. 
So $B^a_{kl}\subsetneq Z$ or $B^a_{kl}\supseteq Z$. 
If $B^a_{kl}\subsetneq Z$, 
then $B^a_{kl}\subseteq C_t$ for some $t\in\{1,...,m\}$.
But since $B^a_{kl}\nsubseteq C$, 
$B^a_{kl}\cap C=\emptyset$, 
a contradition.
Hence we must have $B^a_{kl}\supseteq Z$ when $B^a_{kl}\cap C\neq\emptyset$. 
Thus $\forall a\in A_2$, $\forall k,l\in\mathbb{N}$, 
$B^a_{kl}\cap C=\emptyset$ or $B^a_{kl}\supseteq C$.
\\
\indent
For $a\in A_2$,
if $E(a,C)\neq\emptyset$, 
may assume $C\cap (B^a_{11}\setminus(B^a_{12}\cup...\cup B^a_{1d(1)}))\neq\emptyset$. 
So $C\cap B^a_{11}\neq\emptyset$ and $C\subseteq B^a_{11}$.
For any hole $K\in\{B^a_{12},...,B^a_{1d(1)}\}$, 
if $C\cap K\neq\emptyset$, 
then $C\subseteq K$ and $C\cap B^a_{11}\setminus(B^a_{12}\cup...\cup B^a_{1d(1)}))=\emptyset$, 
a contradiction.
So for any hole $K$,
$C\cap K=\emptyset$.
Thus $C\subseteq
(B^a_{11}\setminus(B^a_{12}\cup...\cup B^a_{1d(1)}))$.
Hence $\forall a\in A_2$, $E(a,C)=\emptyset$ or $E(a,C)=C$.
\end{proof}
So the conclusion holds for $N+1$ (because $|A_2|\geq\frac{1}{2}|X|$ and $|C|\geq\frac{1}{32}|Y|$). 
\end{proof}
\section{Corollary}
\indent 

We can apply theorem \ref{mainthm} to VC-minimal theories (ACVF in particular). The following notions and fact about VC-minimal theories come from \cite{adler2008theories}. We rephrase them as in \cite{cotter2012forking} for notational convenience. 
\begin{definition}
 \cite[Definition 5.]{adler2008theories}
\cite[Definition~2.1.(1)]{cotter2012forking} A set of formulae $\Psi = \{\psi_i(x, \Bar{y_i}) : i \in I\}$ is called a \emph{directed family} if for any
$\psi_0(x, \Bar{y_0}), \psi_1(x, \Bar{y_1}) \in \Psi$ and any parameters $\Bar{a_0}, \Bar{a_1}$ taken from any model
of $T$, one of the following is true:\\
(i): $\psi_0(x, \Bar{a_0}) \subseteq \psi_1(x, \Bar{a_1})$;\\
(ii): $\psi_1(x, \Bar{a_1}) \subseteq \psi_0(x, \Bar{a_0})$;\\
(iii): $\psi_0(x, \Bar{a_0}) \cap \psi_1(x, \Bar{a_1}) = \emptyset$.  
\end{definition}
\begin{definition}
\cite[Definition~3.]{adler2008theories}\cite[Definition~2.1.(2)]{cotter2012forking} A theory $T$ is \emph{VC-minimal} if there is a directed family $\Psi$ such that for any
formula $\varphi(x,\Bar{y})$ and any parameters $\Bar{c}$ taken from any model of $T$, $\varphi(x,\Bar{c})$ is equivalent to a finite boolean combination of formulae $\psi_i(x,\Bar{b_i})$, where each $\psi_i \in \Psi$.
\end{definition}
\begin{fact}
\cite[Proposition~7.]{adler2008theories}\cite[Theorem~2.6.]{cotter2012forking} 
\label{fact}
Fix $T$ a VC-minimal theory and a directed family of formulae $\Psi$ for $T$. For every formula $\tau(x,\Bar{y})$, there are a finite set $\Psi_0 \subseteq \Psi$ and natural numbers $n_1$ and $n_2$ such that for every parameter tuple $\Bar{a}$, $\tau(x,\Bar{a})$ can be decomposed as the union of at most $n_1$ disjoint Swiss cheeses, each of them having at most $n_2$ holes, such that all balls appearing in the decomposition are instances of formulae in $\Psi_0$.
\end{fact}
\indent
By fact \ref{fact}, for any VC-minimal theory $T$, any model $\mathcal{M}\models T$ and any difinable relation $E\subseteq M\times M$, there is $N\in\mathbb{N}^{>0}$ such that for any finite disjoint $X,Y\subseteq M$, the bipartite graph $(X,Y;E)$ has VC-minimal complexity $<N$. Thus we have:
\begin{corollary}
\label{corollary}
    Given a VC-minimal theory $T$, a model $\mathcal{M}\models T$ and an $\mathcal{L}$-formula $\varphi(x,y,\Bar{z})$, let $N\in\mathbb{N}^{>0}$ satisfy: for any $b\in M$, $\Bar{c}\in M^{|\Bar{z}|}$, $\varphi(x,b,\Bar{c})$ can be decomposed as the union of at most $n_1$ disjoint Swiss cheeses, each of them having at most $n_2$ holes, with $n_1n_2<N$. Then for any fixed $\Bar{c}\in M^{|\Bar{z}|}$, any pair of finite sets $X\subseteq M$, $Y\subseteq M$ with $X\cap Y=\emptyset$, there exist $X'\subseteq X$, $Y'\subseteq Y$ such that $|X'|\geq \dfrac{1}{2^{N+4}}|X|$, $Y'\geq\dfrac{1}{2^{N+4}}|Y|$, and $\forall x\in X'$ $\forall y\in Y'$ $\varphi(x,y,\Bar{c})$ or $\forall x\in X'$ $\forall y\in Y'$ $\neg\varphi(x,y,\Bar{c})$. 
    
\end{corollary}
\begin{remark}
     \cite[Example~4.11.(2)]{chernikov2018regularity} shows: Let $\mathcal{M}\models ACVF_{0,0}$ and let a formula $\varphi(x, y,\Bar{z})$ be given. Then there is some $\delta = \delta(\varphi) > 0$ such that for any definable relation $E(x, y) = \varphi(x, y,\Bar{c})$ for some $\Bar{c}\in M^{|\Bar{z}|}$ and finite disjoint $X\subseteq M$, $Y\subseteq M$, there are some $X'\subseteq X$, $Y'\subseteq Y$ with $|X'| \geq\delta|X|$, $|Y'|\geq\delta|Y|$ and $X'\times Y'\subseteq E$ or $X'\times Y'\subseteq\neg E$. By \cite{holly1995canonical}, $ACVF$ has Swiss Cheese decomposition and thus is a VC-minimal theory. So by corollary \ref{corollary}, the same conclusion also holds in $ACVF_{p,q}$ for nonzero $p,q$. (Note: \cite[Example~4.11.(2)]{chernikov2018regularity} allows $|x|>1$ and $|y|>1$ for definable relations $E(x,y)$ in $ACVF_{0,0}$. But we don't know whether the conclusion of theorem \ref{mainthm} holds when $|x|>1$, $|y|>1$ for definable relations $E(x,y)$ in $ACVF_{p,q}$ with nonzero $p,q$, since we only have Swiss cheese decomposition for one-variable formulas in $ACVF$.)
     \end{remark}
     \begin{remark}
     In \cite{chudnovsky2020pure},
     Chudnovsky, Scott, Seymour, and Spirkl proved
 that
     \begin{fact}\cite[1.2.]{chudnovsky2020pure}
     For every forest $H$,
     there exists $\epsilon > 0$ such that for every graph $G$ with $|G| > 1$ that
is both $H$-free and $\overline{H}$-free,
there is a pair of disjoint subsets $(A, B)$ with $|A|, |B| \geq\epsilon|G|$ such that $A\times B\subseteq E$ or $A\times B\cap E=\emptyset$.
 \end{fact}
   The family of forests can be shown to have VC-minimal complexity $\leq 2$ and
   thus the VC-minimal case is not covered in \cite{chudnovsky2020pure}.
   \\
   \indent
   For a forest $H$, $v\in V(H)$,
   let $B_{v,\triangleleft}$ denote the set of the predecessor of $v$ and
   $B_{v,\triangleright}$ 
   denote the set of successors of $v$.
   Consider the family 
   $\mathcal{F}_{H}:=$
   $\{B_{v,\triangleleft},
   B_{v,\triangleright}:
   v\in H\}$.
   $\mathcal{F}_H$ is directed:
   Let $v, w\in V(H)$.
If $B_{v,\triangleleft}\cap
B_{w,\triangleright}\neq \emptyset$, 
then since $B_{v,\triangleleft}$ is a singleton, 
$B_{v,\triangleleft}\subseteq
B_{w,\triangleright}$. 
Similarly, 
if $B_{v,\triangleleft}\cap
B_{w,\triangleleft}\neq \emptyset$,
then
$B_{v,\triangleleft}\subseteq
B_{w,\triangleleft}$. 
If $B_{v,\triangleright}\cap
B_{w,\triangleright}\neq \emptyset$, 
then $v=w$ and 
$B_{v,\triangleright}=
B_{w,\triangleright}$.
\\
\indent
For any $v\in V(H)$,
$E_v=B_{v,\triangleleft}
\sqcup
B_{v,\triangleright}$.
So given any forest $H=(V(H),E)$ and disjoint $X,Y\subseteq V(H)$,
the bipartite graph $(X,Y;E)$ has VC-minimal complexity $\leq 2$.
\end{remark}
\bibliographystyle{alpha}
\bibliography{references}

\end{document}